\documentclass[12pt]{article} 
\usepackage{a4wide, amssymb, amsthm,amsmath}
\usepackage{curves}
\usepackage{hyperref}

\newcommand{\Z}{\mathbb{Z}}
\newcommand{\R}{\mathbb{R}}
\newcommand{\Aut}{{\rm Aut}}
\newcommand{\PSL}{{\rm PSL}}
\newcommand{\lcm}{{\rm lcm}}
\newcommand{\rowstretch}{$\phantom{|^{|^|}_{|}}$}

\newtheorem{theorem}{Theorem}[section]
\newtheorem{lemma}[theorem]{Lemma}

\newtheorem{example}[theorem]{Example}

\title{\Large \bf Soluble quotients of triangle groups} 

\author{
Marston D.E. Conder \\[+4pt] 
{\small Department of Mathematics, University of Auckland} \\[-3pt] 
{\small Private Bag 92019 Auckland, New Zealand} \\[+0pt]  
{\small Email: {\tt m.conder@auckland.ac.nz}} 
\\[+12pt] 
Darius W. Young\\[+4pt] 
{\small Department of Mathematics, University of Auckland} \\[-3pt] 
{\small Private Bag 92019 Auckland, New Zealand} \\[+0pt]  
{\small Email: {\tt darius.young@auckland.ac.nz}} 
}

\date{} 
\begin{document} 
\maketitle 

\begin{abstract}
This paper helps explain the prevalence of soluble groups among the automorphism groups 
of regular maps (at least for `small' genus), by showing that every non-perfect hyperbolic ordinary triangle group 
$\Delta^+(p,q,r) = \langle\, x,y \ | \ x^p = y^q = (xy)^r = 1 \,\rangle$ has a smooth finite soluble quotient of derived 
length $c$ for some $c \le 3$, and infinitely many such quotients of derived length $d$ for every $d > c$. 
\\[-18pt]
\end{abstract}

\section{Introduction} 

For any given integers $p$, $q$ and $r$ greater than $1$, the ordinary $(p,q,r)$ {\em triangle group\/} is the abstract group 
$\Delta^+(p,q,r)$ with presentation $\langle\, x,y,z \ | \ x^p = y^q = z^r = xyz = 1 \,\rangle$, or more simply, 
$\langle\, x,y \ | \ x^p = y^q = (xy)^r = 1 \,\rangle$. 
Such groups and their finite quotients arise naturally in a number of branches of mathematics, most notably in the study of automorphism groups 
of algebraic curves and compact Riemann surfaces, and in the construction and analysis of regular maps on closed surfaces.

More specifically, we note (following some recommendations by a referee) that quotients of triangle groups 
define automorphisms of algebraic curves with algebraic coefficients \cite{Belyi}, and if there exists a 
holomorphic involution inverting the images of two of the three generators, then those coefficients can 
be chosen to be real \cite{KockSingerman}.
They also play a key role in the Grothendieck theory of `dessins d'enfants' (see \cite{GG} for example),
as well as contributing to the study of moduli spaces of complex algebraic curves of given genus, 
in which they allow the construction of isolated points in the Zariski and Teichm{\" u}ller topologies 
-- see \cite{Kulkarni} for example, which initiated a longer series of papers by other authors on this topic. 

Arguably the most famous example is the ordinary $(2,3,7)$ triangle group, which plays a major role in a well-known 
theorem of Hurwitz (1893), stating that a compact Riemann surface of genus $g > 1$ has at most $84(g-1)$ 
conformal automorphisms, with this upper bound attained if and only if the group of all its conformal automorphisms 
is a non-trivial finite quotient of $\Delta^+(2,3,7)$. 
Further details can be found in \cite{Conder-HurwitzSurvey} for example. 

Next, an orientably-regular map ${\cal M} = (X, S)$ is an embedding of a graph or multigraph $X$ into an 
orientable surface $S$, dividing $S$ into simply-connected regions called the faces of the map, 
with the property that the group $G = \Aut^+({\cal M})$ of all orientation- and incidence-preserving automorphisms of ${\cal M}$ 
acts transitively on the set of all arcs (ordered pairs of adjacent vertices) of the graph $X$. 
In any such map, every vertex has the same valency, say $k$, and every face is bounded by a closed walk of the same length, say $m$, 
and then ${\cal M}$ is said to have type $(m,k)$, and $G = \Aut^+({\cal M})$ is a smooth homomorphic image 
of the ordinary $(2,k,m)$ triangle group -- where `smooth' means that the orders $2$, $k$ and $m$ of the generators are preserved, 
or equivalently, that the kernel of the homomorphism from  $\Delta^+(2,k,m)$ onto $G$ is torsion-free. 
Conversely, every smooth homomorphic image $G$ of the ordinary $(2,k,m)$ triangle group gives rise to 
an orientably-regular map ${\cal M}$ of type $(m,k)$ with $G \cong \Aut^+({\cal M})$.  

These facts have been exploited to determine all orientably-regular maps on surfaces of small genus $g > 1$. 
For example, see~\cite{Conder-100} for genus $2$ to $101$, and the associated webpages for genus $2$ to $301$,  
and~\cite{ConderPotocnik} for a very recent extension up to genus $1501$.

The study of regular maps dates back over a century, to work by Dyck and Burnside and others, although some would say 
centuries further back to the determination of the five Platonic solids (which are viewable as regular maps on the sphere). 
Much of the study and construction of such maps concentrated on cases where $\Aut^+({\cal M})$ is a well known group, 
such as the alternating group $A_5$ (or $A_n$ for some larger values of $n$) or $\PSL(2,q)$ for some prime-power $q$, and more recently, 
even some other non-abelian finite simple groups. 

Here we note that a triangle group $\Delta^+(p,q,r)$ is called {\em spherical\/} if $1/p+1/q+1/r > 1$, in which case 
$(p,q,r)$ is some permutation of $(2,2,n)$ for some $n$, or of $(2,3,3)$ $(2,3,4)$ or $(2,3,5)$, 
or {\em Euclidean\/} if $1/p+1/q+1/r = 1$, in which case $(p,q,r)$ is some permutation of $(2,3,6)$ $(2,4,4)$ or $(3,3,3)$, 
and otherwise $\Delta^+(p,q,r)$ is called {\em hyperbolic\/}.  
Also if $p$, $q$ and $r$ are pairwise coprime, then $\Delta^+(p,q,r)$ is perfect, with trivial abelianisation, 
and in that case every finite quotient of $\Delta^+(p,q,r)$ is insoluble. 
For example, $\Delta^+(2,3,7)$ is perfect, which helps explain why its simple quotients have been widely studied. 

A recent (and somewhat surprising) observation by the first author, however, is that the automorphism groups of well over $90\%$ of orientable-regular maps of genus $2$ to $301$ are soluble (while very few are non-abelian simple).  
Moreover, this does not appear to be a `small case only' phenomenon, as it is also the case for such maps of genus $2$ to $1501$.  

The preponderance of these soluble quotients motivated the study behind this paper, which helps explain it. 
In fact we prove the following.

\begin{theorem}
\label{thm:main}
If $\Delta^+$ is any non-perfect hyperbolic ordinary triangle group, then \\[+3pt]
{\rm (a)} $\Delta^+$ has a smooth finite soluble quotient with derived length at most $3$, \\[+3pt]
{\rm (b)} if $c$ is the minimum derived length of a smooth finite soluble quotient of $\Delta^+$, 
then $\Delta^+$ has infinitely many smooth finite soluble quotients with derived length $d$, for every $d > c$. \\[+3pt]
Hence in particular, $\Delta^+$ has infinitely many smooth finite soluble quotients with derived length $d$, for every integer $d > 3$. 
\end{theorem}

This theorem extends one by Sah (see~\cite[Theorem 1.5]{Sah}), showing that every non-perfect hyperbolic Fuchsian group has a smooth soluble quotient of derived length $d$ for any sufficiently large integer $d$, 
and a subsequent one by Zomorrodian, who proved in~\cite{Zomorrodian} that if $\,\Gamma = \Gamma^{(0)} \,\rhd\, \Gamma^{(1)} \,\rhd\, \Gamma^{(2)} \,\rhd\,  \cdots\,$ 
is the derived series of any co-compact non-perfect Fuchsian group $\Gamma$, then $\Gamma^{(c)}$ is a (torsion-free) surface group for some $c \le 4$,  
and applied this to the case of the ordinary $(2,3,n)$ triangle groups, showing that $c \le 3$ for those.
It also extends a very specialised theorem by Chetiya et al  (see~\cite[Theorem 3.1]{ChetiyaDuttaPatra}) concerning ordinary triangle groups 
$\Delta^+(\ell,m,u \ell)$ with $\gcd(\ell,m)=\gcd(u,m)=1$.  
 Our theorem generalises all of these. 
Also a related paper by Gromadzki~\cite{Gromadzki} dealt with 
smooth finite soluble quotients of some particular triangle groups, in obtaining upper bounds on the order of a soluble group 
of conformal automorphisms of a compact Riemann surface of given genus $g > 1$.

In this paper, we will also show that in most cases, the qualification $d > c$ in part (b) can be improved to $d \ge c$, 
and we will give a precise description of the exceptions.

For completeness, we note that (a) the spherical ordinary triangle groups are all finite, and all except $\Delta^+(2,3,5) \cong A_5$ 
are soluble with derived length $1$, $2$ or $3$, while (b) the Euclidean ordinary triangle groups are all infinite and soluble, 
with derived length $2$ (as each has an abelian normal subgroup of rank 2 with cyclic quotient).  
See~\cite[\S 6.4]{CoxeterMoser} for more details about these triangle groups. 

Our proof of Theorem~\ref{thm:main} uses a theorem from 1970 by Singerman \cite{Singerman}, 
which provides a connection between the signatures of a Fuchsian group $\Gamma$ and a subgroup $\Lambda$ 
of finite index in $\Gamma$, via the transitive permutation representation of $\Gamma$ on the (right) cosets of $\Lambda$. 
As the signature encodes some of the group structure, application of Singerman's theorem to the derived series 
of a Fuchsian group is quite fruitful. 

In Section~\ref{sec:Background} we give some further background about Fuchsian groups and their signatures, 
and the Riemann-Hurwitz formula, followed by a statement and some explanation of Singerman's theorem. 
Then in Section~\ref{sec:DerivedDelta} we apply this to determining the signature of the derived group of a non-perfect 
hyperbolic ordinary triangle group, and find there are eight possible cases.  Those cases are then analysed in detail 
in Section~\ref{sec:MainTheoremProof}, resulting in a proof of our main theorem, and finally in Section~\ref{sec:FinalRemarks} 
we make some additional observations and draw some further conclusions that strengthen it.

\section{Further background} 
\label{sec:Background}

A {\em Fuchsian group\/} is a finitely-generated discrete subgroup of $\mathrm{PSL}(2,\R)$. 
Any such group $\Gamma$ has presentation in terms of \\[-22pt] 
\begin{itemize}
    \item $2g$ hyperbolic generators $\{a_1,b_1,\dots,a_g,b_g\}$, \\[-22pt] 
    \item $r$ elliptic generators $\{x_1,x_2,\dots,x_r\}$,  \\[-22pt] 
    \item $s$ parabolic generators $\{p_1,p_2,\dots,p_s\}$, \ and  \\[-22pt] 
    \item $t$ hyperbolic-boundary generators $\{h_1,h_2,\dots,h_t\}$, \\[-22pt] 
\end{itemize}
subject to defining relations: \\[-12pt] 
$$
x_1^{m_1} = x_2^{m_2} = \dots =x_r^{m_r} = \prod_{i=1}^g [a_i,b_i] \cdot \prod_{j=1}^rx_j\cdot\prod_{k=1}^sp_k\cdot\prod_{l=1}^th_l = 1.
$$

The integer $g \ge 0$ is called the {\em genus\/} of $\Gamma$, and the integers $m_1,m_2,\dots,m_r$ (all $\ge 2$) are called the periods of $\Gamma$, 
and we say that $\Gamma$ has Fuchsian {\em signature} $\left(g;\, m_1,m_2,\dots,m_r ;\, s ;\, t  \right)$.
Note for example that an ordinary triangle group $\Delta^+(p,q,r)$ is a Fuchsian group with signature $(0;\, p,q,r;\, 0;\, 0)$.

Next, if $2\pi\mu(\Gamma)$ is the hyperbolic area of a fundamental domain of $\Gamma$, then 
$$
\mu(\Gamma) = 2g-2 +\sum_{i=1}^r \left(1-\frac{1}{m_i}\right) + s + t,  
$$
and if $\Lambda$ is a subgroup of finite index in $\Gamma$, then $\Lambda$ is also a Fuchsian group, 
and the well-known {\em Riemann-Hurwitz formula} gives the index as    
\begin{equation}
|\Gamma:\Lambda| = \frac{\mu(\Lambda)}{\mu(\Gamma)}.
 \label{IRH}
\end{equation}

Singerman's theorem takes this further, as follows:
 
\begin{theorem}{\rm (Singerman \cite{Singerman})}
\label{thm:Singerman}
Let $\Gamma$ be a Fuchsian group with signature $\left(g;\, m_1,\dots,m_r ;\, s ;\, t  \right)$. 
\\[+3pt] 
Then $\Gamma$ contains a subgroup $\Lambda$ of index $N$ with Fuchsian signature 
$$\left( g';\, n_{11},n_{12},\dots,n_{1\rho_1},\dots,n_{r1},n_{r2},\dots,n_{r\rho_r};\, s';\, t'\right)$$ 
if and only if $N = \frac{\mu(\Lambda)}{\mu(\Gamma)}$ and there exists a transitive permutation 
representation $\theta\!:\Gamma \to S_N$ such that \\[+3pt]
{\rm (a)}  the image $x_j^{\ \theta}$ has precisely $\rho_j$ cycles of lengths $\frac{m_j}{n_{j1}},\frac{m_j}{n_{j2}},\dots,\frac{m_j}{n_{j\rho_j}}$, 
       for $1 \le j \le r$, and \\[+3pt]
{\rm (b)}  if $\delta(\gamma)$ denotes the number of cycles of the image $\gamma^{\,\theta}$, for each generator $\gamma$ of $\Gamma$,  then
        $$s'=\sum_{k=1}^s\delta(p_k)  \quad \hbox{and} \quad  t'=\sum_{l=1}^t\delta(h_l).$$
\end{theorem}

Here in the special case where $\Gamma$ is the ordinary triangle group $\Delta^+(p,q,r)$, 
we have just three periods $m_1,m_2,m_3$, and $s = t = 0$, and it follows that $s' = t' = 0$ 
for every such subgroup $\Lambda$ in that case. 

We will be interested only in normal subgroups of finite index, in which case the image $x_j^{\ \theta}$ 
of each elliptic generator will be semi-regular, with all cycles of the same length $\frac{m_j}{n_j}$, say. 
We can also ignore any (elliptic) periods equal to $1$.

The underlying genus $g'$ of $\Lambda$ is given by the Riemann-Hurwitz formula. 

\begin{example}
\label{ex:239} 
Let $\Gamma$ be the ordinary triangle group $\Delta^+(2,3,9)$, 
and let $\Lambda = \Gamma'=[\Gamma,\Gamma]$, the derived subgroup of $\Gamma$, 
with quotient $\Gamma/\Lambda \cong C_3$. 
We can determine the signature of $\Gamma'$ as follows. 
In the natural permutation representation of $\Gamma$ on the three cosets of $\Gamma'$, the images of the 
generators $x_1$, $x_2$ and $x_3$ are the identity permutation, 
a $3$-cycle and its inverse. 
It follows that $\Gamma'$ has three  elliptic generators of orders $n_{11}, n_{12}$ and $n_{13}$, 
plus one of order $n_{21}$ and one of order $n_{31}$, where 
$$
\frac{m_1}{n_{11}} = \frac{m_1}{n_{12}} = \frac{m_1}{n_{13}} = 1, \quad \frac{m_2}{n_{21}} = 3  
\quad \hbox{and} \quad \frac{m_3}{n_{31}} = 3,
$$ 
so $n_1 = n_{11} = n_{12} = n_{13} = 2$, while $n_2 = n_{21} = 1$ and $n_3 = n_{31} = 3$. 
Next,  $\Gamma$ has genus $0$ and therefore 
$\mu(\Gamma) = -2 + (1-\frac{1}{2}) + (1-\frac{1}{3}) + (1-\frac{1}{9}) =  \frac{1}{18}$, 
and then by the Riemann-Hurwitz formula, the genus $\gamma'$ of  $\Gamma'$ is given by 
$$
\frac{3}{18} = |\Gamma:\Gamma'| \, \mu(\Gamma) = \mu(\Gamma') 
= 2g'-2 + \frac{1}{2} + \frac{1}{2} + \frac{1}{2} + \frac{2}{3} = 2g' + \frac{1}{6}. 
$$
Thus $g' = 0$, and finally $($by ignoring the elliptic generator of order $1)$, we see that  
$\Gamma'$ has signature $(0;\, 2,2,2,3;\, 0;\, 0)$. 
\end{example}

We generalise this example in the next section.

\section{\!\!The derived subgroup of a non-perfect triangle group}
\label{sec:DerivedDelta}

We first prove something that was stated without proof in Lemma 8.3 of~\cite{BridsonConderReid}, 
and corrects some unfortunate minor errors (such as a missing `$/$') in the statement of Lemma 2.1 in~\cite{Conder-DTG}. 

\begin{lemma}
\label{lem:abGamma}
Let $\Gamma$ be the ordinary triangle group $\Delta^+(p,q,r)$. 
Then the abelianisation $\Gamma/\Gamma'$ of $\Gamma$ is $C_e \times C_f$, where 
$e =  \lcm(\gcd(p,q),\gcd(q,r),\gcd(p,r))$ which is the exponent, and $f = \gcd(p,q,r)$,  
and $|\Gamma:\Gamma'| = e{\hskip -0.7pt}f = pqr/\lcm(p,q,r)$. 
Moreover, the images in $\Gamma/\Gamma'$ of the three canonical generators 
$x_1$, $x_2$ and $x_3$ of $\Gamma$ 
have respective orders $\,e_1 = \gcd(p,\lcm(q,r))$,  \ $e_2 = \gcd(q,\lcm(p,r)),\,$ and $\,e_3 = \gcd(r,\lcm(p,q))$, 
making $\Gamma/\Gamma'$ a smooth finite abelian quotient of the ordinary triangle group $\Delta^+(e_1,e_2,e_3)$. 
In particular, $\Gamma = \Delta^+(p,q,r)$ has a smooth finite abelian quotient 
if and only if $p$ divides $\lcm(q,r)$, and $q$ divides $\lcm(p,r)$, and $r$ divides $\lcm(p,q)$.
\end{lemma}

\begin{proof}
As $\Gamma$ is $2$-generated, its abelianisation $\Gamma/\Gamma'$ has rank $1$ or $2$. Also clearly 
the largest homocyclic abelian quotient of $\Gamma$ is $C_f \times C_f$ where $f = \gcd(p,q,r)$, 
and hence $\Gamma/\Gamma'$ is isomorphic to $C_e \times C_f$ where $e$ is the exponent of  $\Gamma/\Gamma'$.
The rest of the proof now follows by considering for an arbitrary prime $k$ the largest powers of $k$ dividing $p$, $q$ and $r$, say $k^{\alpha}$, $k^{\beta}$ and $k^{\gamma}$. 
Without loss of generality we may suppose $\alpha \le \beta \le \gamma$, as no linear ordering of $p$, $q$ and $r$ is specified, 
and then $C_{k^{\alpha}} \times C_{k^{\beta}}$ is the largest abelian $k$-quotient of $\Gamma$. 
Accordingly $k^{\beta}$ is the $k$-part of both $e$ and $\lcm(\gcd(p,q),\gcd(q,r),\gcd(p,r))$, 
and  $k^{\alpha+\beta}$ is the $k$-part of both $e{\hskip -0.7pt}f$ and $pqr/\lcm(p,q,r)$,  
while $k^{\alpha}$, $k^{\beta}$ and $k^{\beta}$ are the $k$-parts of $e_1$, $e_2$ and $e_3$, respectively, 
and of $\gcd(p,\lcm(q,r))$,  $\gcd(q,\lcm(p,r))$ and $\gcd(r,\lcm(p,q))$, respectively. 
\end{proof}

Next, for  convenience, we introduce two simplified forms of the signature.
If a Fuchsian group $\Gamma$ has signature $(g;\, m_1,\dots,m_1,\, m_2,\dots,m_2,\dots,\,m_r,\dots,m_r;\, 0;\, 0)$ 
with $n_i$ elliptic generators of order $m_i$, for $1 \le i \le r$, and $s = t = 0$, we can first abbreviate 
the signature of $\Gamma$ to $(g;\, m_1^{(n_1)},m_2^{(n_2)},\dots,m_r^{(n_r)})$, 
and then replace each term $m_i^{(1)}$ with $n_i = 1$ by just $m_i$, to obtain what we will call 
the {\em compact signature\/} of $\Gamma$. 
For instance, the compact signature $(g;\, -)$ indicates a Fuchsian group with no non-trivial elliptic generators, 
while in Example~\ref{ex:239}, the signature $(0;\, 2,2,2,3;\, 0;\, 0)$ is abbreviated to $(0;\, 2^{(3)},3^{(1)})$ 
and then its compact form is $(0;\, 2^{(3)},3)$.  

We can now determine all possible compact signatures for the derived subgroup of a non-perfect triangle group. 
\medskip

\begin{theorem}
\label{thm:SigTypes}
Let $\Gamma=\Delta^+(p,q,r)$ be a non-perfect ordinary triangle group. 
Then the compact signature of its derived group $\Gamma'=[\Gamma,\Gamma]$ has one of the eight forms 
listed in the first column of Table~{\rm \ref{table:SigTypes}},  
where the $m_i$ are pairwise coprime, and each $n_i$ is greater than $1$. 
In particular, $\Gamma'$ is not perfect.
\end{theorem}
    
\begin{table}[ht]
\begin{center}
\begin{tabular}{|| l | l | l ||} 
\hline
Compact signature of $\Gamma'$  & Example $\Gamma$ & Compact signature of $\Gamma'$ for the example  \\[0.5ex] 
 \hline\hline  
 $\ (g;\, -)$ \rowstretch & $\ \Delta^+(4,4,4)$ & $\ (3; - )$  \\[+2pt] 
 \hline
 $\ (g;\, m_1)$ \rowstretch & $\ \Delta^+(2,3,12)$  & $\ (1;\,2)$  \\
 \hline
 $\ (g;\, m_1^{(n_1)})$ \rowstretch &  $\ \Delta^+(3,3,4)$ & $\ (0;\, 4^{(3)})$  \\
 \hline
 $\ (g;\, m_1,m_2^{(n_2)})$ \rowstretch & $\ \Delta^+(2,3,8)$ & $\ (0;\, 4,3^{(2)})$  \\
 \hline
 $\ (g;\, m_1^{(n_1)},m_2^{(n_2)})$ \rowstretch & $\ \Delta^+(2,4,6)$ & $\ (0;\, 2^{(2)},3^{(2)})$  \\ [1ex] 
 \hline
 $\ (g;\, m_1,m_2,m_3^{(n_3)})$ \rowstretch &  $\ \Delta^+(2,9,15)$ & $\ (0;\, 3,5,2^{(3)})$  \\ [1ex] 
 \hline
 $\ (g;\, m_1,m_2^{(n_2)},m_3^{(n_3)})$ \rowstretch &  $\ \Delta^+(4,9,30)$ & $\ (1;\, 5,2^{(3)},3^{(2)})$  \\ [1ex]    
 \hline
 $\ (g;\, m_1^{(n_1)},m_2^{(n_2)},m_3^{(n_3)})$ \rowstretch \ & $\ \Delta^+(4,6,10)$ &  $\ (0;\, 2^{(2)},3^{(2)},5^{(2)})$  \\ [1ex]    
 \hline
\end{tabular}
\caption{Compact signatures of derived groups of non-perfect triangle groups}
\label{table:SigTypes}
\end{center}
\end{table}

\begin{proof}
First, we let $x_1$, $x_2$ and $x_3$ be  canonical generators for $\Gamma=\Delta^+(p,q,r)$, of orders $p$, $q$ and $r$ respectively, 
with $x_1 x_2 x_3 = 1$.

Next, by Singerman's Theorem (Theorem~\ref{thm:Singerman}), the abbreviated form 
of the signature of $\Gamma'$ must be $(g; -)$, $(g;\, m_1^{(n_1)})$, $(g;\, m_1^{(n_1)},m_2^{(n_2)})$ or  $(g;\, m_1^{(n_1)},m_2^{(n_2)},m_3^{(n_3)})$ 
for some $g \ge 0$ and some $m_1,m_2,m_3 > 1$, and these four possibilities give $1$, $2$, $3$ and $4$ potential candidates 
for the compact signature, respectively.  
We will eliminate the candidates  $(g;\, m_1,m_2)$ and $(g;\, m_1,m_2,m_3)$, leaving only the eight possibilities listed in Table~\ref{table:SigTypes}. 

To do this, we consider the transitive permutation representation of $\Gamma$ on the $N$ cosets of $\Gamma'$, 
where $1 < N = |\Gamma:\Gamma'| = e{\hskip -0.7pt}f = pqr/\lcm(p,q,r)$ as noted earlier. 
By Lemma~\ref{lem:abGamma}, the orders of the images of the permutations induced by $x_1$, $x_2$ and $x_3$ 
are $\,e_1 = \gcd(p,\lcm(q,r))$,  \ $e_2 = \gcd(q,\lcm(p,r)),\,$ and $\,e_3 = \gcd(r,\lcm(p,q)),\,$ 
so the permutation induced by $x_i$ has $n_i = \frac{N}{e_i}$ cycles of length $e_i$ for $1 \le i \le 3$. 
Furthermore, $\,m_1 = \frac{p}{e_1}$, \ $m_2 = \frac{q}{e_2}\,$ and $\,m_3 = \frac{r}{e_3},\,$ to give the abbreviated signature of $\Gamma'$. 

Next, if $k$ is a prime divisor of $m_1 = \frac{p}{e_1} = \frac{p}{\gcd(p,\lcm(q,r))}$, 
then a higher power of $k$ divides $p$ than divides $\lcm(q,r)$, so $k$ cannot divide either $e_2 = \gcd(q,\lcm(p,r))$ or $e_3 = \gcd(r,\lcm(p,q))$, 
and hence $m_1$ is coprime to each of $m_2$ and $m_3$.  
An analogous argument shows that $m_2$ is coprime to $m_1$ and $m_3$, and so the $m_i$ are pairwise coprime.

We now show that the compact signature of $\Gamma'$ cannot be $(g;\, m_1,m_2,m_3)$. 
For assume the contrary.  Then each $n_i = \frac{N}{e_i}$ is $1$, so $N = e_1 = e_2 = e_3$ 
and therefore $N$ divides each of $p$, $q$ and $r$, and hence divides both $e =  \lcm(\gcd(p,q),\gcd(q,r),\gcd(p,r))$ and $f = \gcd(p,q,r)$. 
Thus $N^2$ divides $e{\hskip -0.7pt}f = N$, so $N = 1$, a contradiction. 

Finally we show that $\Gamma'$ cannot have compact signature $(g;\, m_1,m_2)$. For again assume the contrary.
Then without loss of generality, $N = e_1 = e_2$ while $e_3 = r$.  In this case $N$ divides each of $p$ and $q$, 
and hence divides $\gcd(p,q)$ which divides $e$.  Then since $N = e{\hskip -0.7pt}f$ we find that $N = \gcd(p,q) = e$ and so $f = 1$,  On the other hand, $r = e_3$ which divides $N= \gcd(p,q)$, and so $f = \gcd(p,q,r) > 1$, 
another contradiction, completing the proof. 
\end{proof}

We leave it as an exercise for the reader to confirm that the examples in the second column of Table~{\rm \ref{table:SigTypes}} 
give the compact signatures in the third column.

\section{Proof of our main theorem}
\label{sec:MainTheoremProof}

We prove part (a) of Theorem~\ref{thm:main}, followed by part (b).  

First let $\Gamma = \Delta^+(p,q,r)$ be any non-perfect hyperbolic ordinary triangle group.  
We now show that $\Gamma$ has a smooth finite soluble quotient with derived length $1$, $2$ or $3$, 
depending on the compact signature of its derived subgroup $\Gamma'$.
We do this in cases, numbered by the relevant row of Table~{\rm \ref{table:SigTypes}}, 
but not in exactly the given order.  

We will make repeated use of the following fact about an arbitrary group $G$ generated by 
a finite subset of size $n$. 
If $m$ is any positive integer, and $G^{(m)}$ is the characteristic subgroup of $G$ 
generated by the $m\,$th powers of all elements of $G$, then the product $K_m(G) = G' \,G^{(m)}$ 
is also a characteristic subgroup of $G$, with abelian quotient $G/K_m(G)$ being 
a factor group of the direct product $(C_m)^{n}$.  

Here we note (again following a recommendation by a referee) that aspects of Cases (3), (5) and (8) 
could also be dealt with by considering a smooth homomorphism from $\Gamma$ onto a cyclic group 
of order $m_1$, $m_{1}m_{2}$ or $m_{1}m_{2}m_{3}$ respectively, as explained in \cite{Harvey}. 

${}$\\[-24pt]

{\bf Case (1)}:  \ \ $\Gamma'$ has compact signature $(g; -)$.

This case is easy: $\Gamma'$ is torsion-free, and so $\Gamma/\Gamma'$ is a smooth soluble quotient 
with derived length $1$. 

${}$\\[-30pt]

Now we skip to case (3) before returning to case (2).

${}$\\[-24pt]

{\bf Case (3)}:  \ \ $\Gamma'$ has compact signature $(g;\, m_1^{(n_1)})$.

In this case $\Gamma'$ has presentation \\[-18pt] 
$$
\langle\, a_1,b_1,\dots,a_g,b_g, \, x_1, \dots, x_{n_1} \ | \ x_1^{m_1} = \dots = x_{n_1}^{m_1} = x_1\dots x_{n_1} \prod_{i=1}^{2g}\, [a_i,b_i] = 1 \,\rangle. 
$$
Abelianisation gives $\Gamma'/\Gamma'' \cong \Z^{2g} \times (C_{m_1})^{n_1-1}$, 
and then if $m = m_1$ (or any integer multiple of $m_1$), then reduction mod $m$ gives  
$\Gamma'/K_{m}(\Gamma') \cong (C_{m})^{2g} \times (C_{m_1})^{n_1-1},$
and then since $K_{m}(\Gamma')$ contains no element $x_i^{\,h}$ for $1 \le h < m_1$, 
it follows that $K_{m}(\Gamma')$ is torsion-free. 
Also $K_{m}(\Gamma')$ is characteristic in $\Gamma'$ and hence normal in $\Gamma$, 
and therefore $\Gamma/K_{m}(\Gamma')$ is a smooth finite soluble quotient of $\Gamma$, with derived length $2$.

${}$\\[-24pt]

{\bf Case (2)}:  \ \ $\Gamma'$ has compact signature $(g;\, m_1)$.

Here $g\geq 1$, for otherwise $\Gamma'$ and hence $\Gamma$ would be finite. 
In this case $\Gamma'$ has presentation \\[-12pt] 
$$
\langle\, a_1,b_1,\dots,a_g,b_g, \, x_1 \ | \ x_1^{m_1} = x_1 \prod_{i=1}^{2g}\, [a_i,b_i] = 1 \,\rangle,  
$$
with $\Gamma'/\Gamma''\cong \Z^{2g}$, and as $x_1 = \left (\prod_{i=1}^{g}[a_i,b_i] \right )^{-1} \in \Gamma''$, 
we find that $x_1$ cannot have order $m_1$ in any quotient of $\Gamma'/\Gamma''$, 
and hence $\Gamma$ has no smooth quotient with derived length $2$. 

But now for any integer $m > 1$ coprime to $m_1$, we can apply Singerman's theorem to the finite-index subgroup $K_{m}(\Gamma')$ of $\Gamma'$. 
As $x_1 \in \Gamma''  \le  K_m(\Gamma')$, we see that $\Gamma'/K_{m}(\Gamma') \cong (C_{m})^{2g}$, 
with $x_1$ inducing the identity permutation. 
Hence $K_{m}(\Gamma')$ has compact signature $(g';\,  m_1^{(N)})$ for some $g'$, with $N = m^{\,2g}$. 
This can be handled in the same way as in case (3), by letting $m' = m_1$ (or any integer multiple of $m_1$) 
and then taking $L=K_{m'}(K_{m}(\Gamma'))$ which is characteristic 
in $K_{m}(\Gamma')$ and hence normal in $\Gamma$, and gives a smooth finite soluble quotient $\Gamma/L$ with derived length $3$.

${}$\\[-24pt]

{\bf Case (4)}:  \ \ $\Gamma'$ has compact signature $(g;\, m_1,m_2^{(n_2)})$.

In this case $\Gamma'$ has presentation \\[-15pt] 
$$
\langle\, a_1,b_1,\dots,a_g,b_g, \, x_1,\,  y_1, \dots, y_{n_2}  \ | \  x_1^{m_1}=y_1^{m_2}=\dots=y_{n_2}^{m_2} = x_1\,y_1 \dots y_{n_2}\,\prod_{i=1}^{2g}\, [a_i,b_i]=1 \,\rangle.
$$
Here $x_1 = \left (y_1\dots y_{n_2}\,\prod_{i=1}^{g}[a_i,b_i] \right )^{-1}$, so $\Gamma'$ is generated by $a_1,b_1,\dots,a_g,b_g, \,  y_1, \dots, y_{n_2}$, 
and therefore $\Gamma'/\Gamma''$ is isomorphic to $\Z^{2g} \times A$ 
for some quotient $A$ of $(C_{m_2})^{n_2}$.  In particular, this abelian group contains 
no element with order a non-trivial divisor of $m_1$, because $\gcd(m_1,m_2) = 1$ and $\Z^{2g}$ is torsion-free, and it follows that $x_1 \in \Gamma''$,
and therefore $\Gamma$ has no smooth quotient with derived length $2$. 

Now consider the natural presentation of the quotient $\Gamma'/\Gamma''$ in terms of the images of its  generators. 
Letting $\overline{w}$ be the image of any element $w$ of $\Gamma'$, 
we note that  $\overline{x_1} = 1$ (because $x_i \in \Gamma''$), 
and that $\prod_{i=1}^{2g}\, [\overline{a_i},\overline{b_i}] = 1$ (because $\Gamma'/\Gamma''$ is abelian), 
and then since $x_1\,y_1 \dots y_{n_2}\,\prod_{i=1}^{2g}\, [a_i,b_i]=1$ 
it follows that $\overline{y_1 \dots y_{n_2}} = 1$, 
so $\overline{y_{n_2}} = (\overline{y_1}\dots \overline{y_{n_2-1}})^{-1}$, 
and therefore $\Gamma'/\Gamma'' \cong \Z^{2g} \times (C_{m_2})^{n_2-1}$.  

Next, if $m = m_2$ (or any integer multiple of $m_2$ coprime to $m_1$), then reduction mod $m$ gives $\Gamma'/K_{m}(\Gamma')\cong (C_{m})^{2g} \times (C_{m_2})^{n_2-1}$, 
and then application of Singerman's theorem shows that the finite index subgroup $K_{m}(\Gamma')$ has compact signature $(g';\, m_1^{(N)})$ 
for some $g'$, with $N = m^{\,2g}m_2^{\,n_2-1}$. 
This can be handled in a similar way to case (2),  by letting $m' = m_1$ (or any integer multiple of $m_1$) 
and then taking $L=K_{m'}(K_{m}(\Gamma'))$ which is characteristic in $K_{m}(\Gamma')$ and hence normal in $\Gamma$, 
and therefore gives a smooth finite soluble quotient $\Gamma/L$ with derived length $3$.

${}$\\[-24pt]

{\bf Case (5)}:  \ \ $\Gamma'$ has compact signature $(g;\, m_1^{(n_1)},m_2^{(n_2)})$. 

Here $\Gamma'$ has presentation \\[+6pt] 
${}$ \quad 
$\langle\, a_1,b_1,\dots,a_g,b_g, \, x_1, \dots, x_{n_1},\,  y_1, \dots, y_{n_2}  \ | \  x_1^{m_1}=\dots=x_{n_1}^{m_1}=y_1^{m_2}=\dots=y_{n_2}^{m_2}$ \\[+3pt]
${}$ \hskip 3.3in $\displaystyle = x_1\dots x_{n_1}\,y_1\dots y_{n_2}\,\prod_{i=1}^{2g}\, [a_i,b_i]=1 \,\rangle,$
\\[+6pt]
and this time we find that $\Gamma'/\Gamma''$ is generated by $2g$ elements of infinite order, 
plus $n_1$ elements of order $m_1$, and $n_2$ elements of order $m_2$, whose product is the identity element.  
Letting $\overline u$, $\overline v$ and $\overline w$ be the images in $\Gamma'/\Gamma''$ of $u = x_1 \dots x_{n_1},\,$ 
$v = y_1 \dots y_{n_2}$ and $w = \prod_{i=1}^{2g}\, [a_i,b_i]$, with $uvw = 1$, we note that $\overline u \overline v$ has order $m_1 m_2$ while $\overline w$ is trivial, 
and so $\overline u \overline v = \overline u \overline v \overline w = 1$, 
and then since $\overline u$ and $\overline v$ have coprime finite orders $m_1$ and $m_2$, also $\overline u = \overline v = 1$.  
Hence the images of $x_{n_1}$ and $y_{n_2}$ are the inverses of the images of $x_1 \dots x_{n_1-1}$ and $y_1 \dots y_{n_2-1}$, 
respectively, and it follows that $\Gamma'/\Gamma'' \cong \Z^{2g} \times (C_{m_1})^{n_1-1} \times (C_{m_2})^{n_2-1}$.  

Now taking $m = m_1 m_2$ (or any integer multiple of $m_1 m_2$), we find that $\Gamma'/K_m(\Gamma')\cong  (C_m)^{2g} \times (C_{m_1})^{n_1-1} \times (C_{m_2})^{n_2-1}$, 
and therefore $\Gamma/K_m(\Gamma')$ is a smooth finite soluble quotient of $\Gamma$ with derived length $2$.

${}$\\[-24pt]

{\bf Case (6)}:  \ \ $\Gamma'$ has compact signature $(g;\, m_1,m_2,m_3^{(n_3)})$.

The argument in this case is similar to the one in cases (4) and (5).  
Here  $\Gamma'$ is generated by the $2g+2+n_3$ elements  $a_1,b_1,\dots,a_g,b_g, \, x_1,\, y_1,\, z_1, \dots, z_{n_3}$
subject to the defining relations \\[-12pt] 
$$
x_1^{m_1} = y_1^{m_2} = z_1^{m_3} = \dots = z_{n_3}^{m_3} 
= x_1\,y_1\,z_1 \dots z_{n_3}\,\prod_{i=1}^{2g}\, [a_i,b_i]=1, 
$$
and since $m_1,m_2$ and $m_3$ are pairwise coprime we find that the images of $x_1$ and $y_1$ in $\Gamma'/\Gamma''$ are trivial, 
so $\Gamma$ has no smooth quotient with derived length $2$. 
But also the image of $z_{n_3}$ is equal to the inverse of the image of $z_1 \dots z_{n_3-1}$, 
and so $\Gamma'/\Gamma'' \cong \Z^{2g} \times (C_{m_3})^{n_3-1}$.

Now if $m = m_3$ (or any multiple of $m_3$ coprime to both $m_1$ and $m_2$), then $\Gamma'/K_{m}(\Gamma')\cong  (C_{m})^{2g} \times (C_{m_3})^{n_3-1}$,  
and application of Singerman's theorem tells us that $K_{m}(\Gamma')$ has compact signature $(g';\, m_1^{(N)},m_2^{(N)})$ for some $g'$, with $N = m^{\,2g} m_3^{\,n_3-1}$. 
Hence by letting $m' = m_1 m_2$ (or any integer multiple of $m_1 m_2$) 
and then taking $L=K_{m'}(K_{m}(\Gamma'))$, we obtain a smooth finite soluble quotient $\Gamma/L$ of $\Gamma$ with derived length $3$.

${}$\\[-24pt]

{\bf Case (7)}:  \ \ $\Gamma'$ has compact signature $(g;\, m_1,m_2^{(n_2)},m_3^{(n_3)})$.

Again this case is similar to earlier ones. 
Here  $\Gamma'$ is generated by the $2g+1+n_2+n_3$ elements  $a_1,b_1,\dots,a_g,b_g, \, x_1,\, y_1,\dots, y_{n_2},\, z_1, \dots, z_{n_3}$
subject to defining relations \\[-12pt] 
$$
x_1^{m_1} = y_1^{m_2} = \dots = y_{n_2}^{m_2} = z_1^{m_3} = \dots = z_{n_3}^{m_3} 
= x_1\,y_1 \dots y_{n_2} \,z_1 \dots z_{n_3}\,\prod_{i=1}^{2g}\, [a_i,b_i]=1, 
$$
and we find that the image of $x_1$  in $\Gamma'/\Gamma''$ is trivial, so $\Gamma$ has no smooth quotient with derived length $2$. 
But also the images of $y_{n_2}$ and $z_{n_3}$ are equal to the inverses of the images of $y_1 \dots y_{n_2-1}$ and $z_1 \dots z_{n_3-1}$ respectively, 
and so $\Gamma'/\Gamma'' \cong \Z^{2g} \times (C_{m_2})^{n_2-1} \times (C_{m_3})^{n_3-1}$.  

It follows that if $m = m_2 m_3$ (or any multiple of $m_2 m_3$ coprime to $m_1$), 
then $\Gamma'/K_{m}(\Gamma')\cong  (C_{m})^{2g} \times  (C_{m_2})^{n_2-1} \times (C_{m_3})^{n_3-1}$, 
and application of Singerman's theorem tells us that $K_{m}(\Gamma')$ has compact signature $(g';\, m_1^{(N)})$ for some $g'$, with $N = m^{\,2g} m_2^{\,n_2-1}\, m_3^{\,n3-1}$.  
Hence by letting $m' = m_1$ (or any integer multiple of $m_1$) and then taking $L=K_{m'}(K_{m}(\Gamma'))$, 
we obtain a smooth finite soluble quotient $\Gamma/L$ of $\Gamma$ with derived length $3$.

${}$\\[-24pt]

{\bf Case (8)}:  \ \ $\Gamma'$ has compact signature $(g;\, m_1^{(n_1)},m_2^{(n_2)},m_3^{(n_3)})$.

This final case is similar to cases (3) and (5).  Here  $\Gamma'$ is generated by the $2g+n_1+n_2+n_3$ elements  
$a_1,b_1,\dots,a_g,b_g, \, x_1,\dots,x_{n_1},\, y_1,\dots,y_{n_2},\, z_1, \dots, z_{n_3}$, 
subject to defining relations \\[-2pt] 
$$ {}\hskip 3pt 
x_1^{m_1} \!= \dots =\! x_{n_1}^{m_1} = y_1^{m_2} \!= \dots =\! y_{n_2}^{m_2} = z_1^{m_3} \!= \dots =\! z_{n_3}^{m_3} 
$$
$$
= x_1 \dots x_{n_1} \,y_1 \dots y_{n_2} \,z_1 \dots z_{n_3}\,\prod_{i=1}^{2g}\, [a_i,b_i]=1,  
$$
and then $\Gamma'/\Gamma'' \cong \Z^{2g} \times (C_{m_1})^{n_1-1} \times (C_{m_2})^{n_2-1} \times (C_{m_3})^{n_3-1}$ 
by similar arguments to those used earlier.  
Hence for any multiple $m$ of $m_1 m_2 m_3$, the group $\Gamma/K_{m}(\Gamma')$ is a smooth finite soluble quotient of $\Gamma$ with derived length $2$.

This completes the proof of part (a) of our main theorem. 

\medskip

We now proceed to prove part (b), namely that if $c$ is the minimum derived length of smooth finite soluble quotient of $\Delta^+$, 
then $\Delta^+$ has infinitely many smooth finite soluble quotients with derived length $d$, for every integer $d > c$. 

To do this, we make use of the so-called `Macbeath trick', which was first introduced in~\cite{Macbeath1961}, and goes as follows. 

Suppose the finite group $G$ is a smooth finite quotient of the hyperbolic ordinary triangle group $\Gamma = \Delta^+(p,q,r)$. 
Then $G$ is a group of conformal automorphisms of a compact Riemann surface $S$ of genus $g$, where \\[-12pt] 
$$2-2g = |G|\left(\frac{1}{p}+\frac{1}{q}+\frac{1}{r}-1\right),$$  
in accordance with the Riemann-Hurwitz formula.
The kernel $K$ of the corresponding smooth homomorphism from $\Gamma$ onto $G$ is the fundamental group
of the surface $S,$ and is itself a Fuchsian group, with signature $(g; - )$.  
In particular, if $g \ge 1$ then $K$ is generated by $2g$ elements $a_1, b_1,\ldots, a_g, b_g$ subject to a single defining relation 
$[a_1,b_1] \ldots [a_g,b_g] = 1$. 

Now for any positive integer $m$, the subgroup $K_m(K)$ is characteristic in $K$ and hence normal in $\Gamma$, 
and the quotient $\Gamma/K_m(K)$ is then isomorphic to an extension by $G$ of a normal subgroup $K/K_m(K) \cong (C_m)^{2g}$ of rank $2g$ 
and exponent $m$. 
In particular, $\Gamma/K_m(K)$ is finite and smooth, and if $G \cong \Gamma/K$ is soluble, then so is $\Gamma/K_m(K)$. 

Moreover, $\Gamma/K_m(K)$ is a group of conformal automorphisms of a compact Riemann surface $S$ of genus $g'$, where \\[-12pt] 
$$2-2g' = |\Gamma/K_m(K)|\left(\frac{1}{p}+\frac{1}{q}+\frac{1}{r}-1\right)$$  
and as  $ |\Gamma/K_m(K)| = m^{2g}|G| > |G|$ while $\frac{1}{p}+\frac{1}{q}+\frac{1}{r}-1 < 0$, it follows that $2-2g > 2-2g'$ and hence $g' > g$.
\smallskip

So now construct an infinite sequence $G = G_0, G_1, G_2, \dots$, where each $G_i$ is a finite smooth quotient of $\Gamma$ 
obtained from the previous one by the Macbeath trick, with $G_i$ containing a finite abelian normal subgroup $N_i$ of rank $2g_{i-1}$,  
exponent $m_i$ and order $m_i^{\,2g_{i-1}}$, such that $G_i/N_i \cong G_{i-1}$. 
Then since $g = g_1 < g_2 < \cdots$, it is easy to see that when $G_{i-1}$ is soluble with derived length $d$, 
the new smooth quotient $G_i$ is soluble with derived length $d+1$. 

Finally, because there are infinitely many possible choices for each $m_i$, it follows that we have infinitely many smooth finite quotients of $\Gamma$ 
with derived length $d$, for every $d$ greater than the derived length of $G$.  

This completes part (b). 

The final claim in Theorem~\ref{thm:main}, namely that  $\Gamma = \Delta^+(p,q,r)$ has infinitely many 
smooth finite soluble quotients with derived length $d$, for every integer $d > 3$, is an immediate consequence of parts (a) and (b). 
\hfill $\square$

\section{Final remarks}
\label{sec:FinalRemarks}

It is easy to see that a smooth finite soluble quotient with derived length $1$ 
(or equivalently, a smooth finite abelian quotient) 
exists only in Case (1), where $\Gamma'$ has compact signature $(g; -)$ for some $g > 0$.
This occurs if and only if each of the triangle group parameters $p,q$ and $r$ divides the least common multiple of the other two, 
by Lemma~\ref{lem:abGamma}. 

Also we see that a smooth finite soluble but non-abelian quotient with derived length $2$ exists only in Cases (3), (5) and (8), 
which are precisely the cases having the property that each $m_i > 1$ in the unabbreviated signature of $\Gamma'$ occurs with multiplicity $n_i > 1$. 

In the other four cases, some $m_i > 1$ in the signature of $\Gamma'$ occurs with multiplicity $n_i = 1$, 
and then the image of the corresponding elliptic generator of $\Gamma'$ lies in $\Gamma''$, 
and hence the minimum derived length of a a smooth finite soluble quotient of $\Gamma$  is $3$.  
\bigskip

Next, in some cases part (b) of Theorem~\ref{thm:main} can be improved, to say there are infinitely many smooth finite soluble 
quotients with derived length $d$, for every $d \ge c$, rather than just for every $d > c$.

In particular, suppose the compact signature of some term $\Gamma_i$ of a normal series for $\Gamma$ used in the proof of part (a) 
to find a smooth finite soluble quotient with derived length $c$ has positive genus parameter ($g$ or $g'$).
Then there are 
infinitely many possibilities for the exponent of the corresponding factor $\Gamma_i/\Gamma_{i+1}$ of the series, 
and these give infinitely many possibilities for the smooth finite quotient of derived length $c$. 

More specifically, we have the following:

\begin{theorem}
Let $\Gamma = \Delta^+(p,q,r)$ be a non-perfect hyperbolic ordinary triangle group, with the signature of $\Gamma'$ 
as given in Table~{\rm \ref{table:SigTypes}}, 
and let $c$ be the minimum derived length of a smooth finite soluble quotient of $\Gamma$. 
Then the number of  smooth finite soluble quotients of $\Gamma$ with derived length equal to $c$ is 
\\[-24pt] 
\begin{itemize}
\item[{\rm (a)}] finite in Case $(1);$  \\[-21pt] 
\item[{\rm (b)}] infinite in Cases $(2)$ and $(7);$ \\[-21pt] 
\item[{\rm (c)}] finite in Cases $(3)$, $(5)$ and $(8)$ when $g = 0$, but infinite when $g > 0;$ \\[-21pt] 
\item[{\rm (d)}] finite in Cases $(4)$ and $(6)$ when there is some permutation of $p$, $q$ and $r$ such that 
$r$ is coprime to both $p$ and $q$, and $p \ne q$, and either {\rm (i)} $\gcd(p,q) = 2$, or {\rm (ii)} $r = 2$ and  $\gcd(p,q) = 3$, 
but otherwise infinite.  
\end{itemize}
\end {theorem}

The situation is summarised in Table~\ref{table:SigTypeLengths}. 
\smallskip

\begin{table}[ht]
\begin{center}
\begin{tabular}{|| l | c | c ||} 
 \hline
Compact signature of $\Gamma'$  & $\ \ c \ \ $ & Infinitely many with DL $= c$?\\[0.5ex] 
 \hline\hline  
 $\ (g; -)$ \rowstretch & $1$ &  Never \\[+2pt] 
 \hline
 $\ (g;\, m_1)$ \rowstretch & $3$  & Always \\
 \hline
 $\ (g;\, m_1^{(n_1)})$ \rowstretch &  $2$ & Whenever $g > 0$ \\
 \hline
 $\ (g;\, m_1,m_2^{(n_2)})$ \rowstretch & $3$ & Depends on $p$, $q$ and $r$   \\
 \hline
 $\ (g;\, m_1^{(n_1)},m_2^{(n_2)})$ \rowstretch & $2$ & Whenever $g > 0$\\ [1ex] 
 \hline
 $\ (g;\, m_1,m_2,m_3^{(n_3)})$ \rowstretch &  $3$  & Depends on $p$, $q$ and $r$ \\ [1ex] 
 \hline
 $\ (g;\, m_1,m_2^{(n_2)},m_3^{(n_3)})$ \rowstretch &  $3$ & Always \\ [1ex]    
 \hline
 $\ (g;\, m_1^{(n_1)},m_2^{(n_2)},m_3^{(n_3)})$ \rowstretch \ & $2$ &  Whenever $g > 0$ \\ [1ex]    
 \hline
\end{tabular}
\caption{Minimum derived length $c$ of a smooth finite soluble quotient}
\label{table:SigTypeLengths}
\end{center}
${}$\\[-36pt]
\end{table}

\begin{proof}
In Case (1), where $c = 1$, every smooth finite soluble quotient with derived length $c$ 
is a quotient of the finite abelianisation of $\Gamma$, so there are only finitely many such quotients 
(and sometimes only one). 

On the other hand, in Case (2) there are always infinitely many smooth finite soluble quotients with derived length $c = 3$, 
because  $\Gamma'$ has compact signature $(g;\, m_1)$ with $g \ge 1$, and then for any integer $m > 1$ coprime to $m_1$, 
we showed in the proof of Theorem~\ref{thm:main}  that $\Gamma/(K_{m_1}(K_{m}(\Gamma'))$ is a smooth finite soluble 
quotient of $\Gamma$ with derived length $3$ and order divisible by $|\Gamma'/K_{m}(\Gamma')| = m^{2g}$,  
so the order of $\Gamma/(K_{m_1}(K_{m}(\Gamma'))$ is unbounded. 

In Cases (3), (5) and (8), there exist infinitely many smooth finite soluble quotients with derived length $c = 2$  
when $g > 0$, but otherwise (if $g = 0$) then $|\Gamma'/\Gamma''|$ is bounded above by $m_1^{\,n_1-1}$, $m_1^{\,n_1-1}\,m_2^{\,n_2-1}$ 
and $m_1^{\,n_1-1}\,m_2^{\,n_2-1}\,m_3^{\,n_3-1}$ respectively, so there are only finitely many. 

(For example, when $\Gamma = \Delta^+(3,3,6)$ the derived subgroup $\Gamma'$ has compact signature $(1;\, 2,2,2)$, 
and there are infinitely many smooth finite soluble quotients with derived length $2$, but this does not 
happen when $\Gamma = \Delta^+(3,3,4)$, because for that group, $\Gamma'$ has compact signature $(0;\, 4,4,4)$ 
and there is exactly one smooth finite soluble quotient with derived length $2$, namely $\Gamma/\Gamma''$, which has order $48$.)  

For the remainder of the proof, we suppose that one of Cases (4), (6) and (7) applies. 

Then if the parameter $g$ in the compact signature of $\Gamma'$ is positive, then as shown in our proof of Theorem~\ref{thm:main}, 
there are infinitely many possible choices for a positive integer $m$ such that there exists a smooth finite soluble quotient $\Gamma/L$ 
of $\Gamma$ with derived length $3$ for which $|\Gamma'/L|$ is multiple of $m^{2g}$ by some constant depending on the $m_i$ and $n_i$,  
and hence $\Gamma$ has infinitely many such quotients.  So from now on, let us suppose that $g = 0$.   
\smallskip

Then letting $n = |\Gamma:\Gamma'|$, we know that $\Gamma'$ has abbreviated signature $(0;\ m_1^{(n_1)}, m_2^{(n_2)}, m_3^{(n_3)})$  
where $(m_1,m_2,m_3) = (\frac{p}{e_1},\frac{q}{e_2},\frac{r}{e_3})$, and $n_i = \frac{n}{e_i}$ for $i = 1,2,3$. 
It follows easily that 
$$n\left (1 - \frac{1}{p} - \frac{1}{q} - \frac{1}{r} \right ) = n\mu(\Gamma) = \mu(\Gamma') 
= 0-2 + \frac{n}{e_1}(1-\frac{e_1}{p})  + \frac{n}{e_2}(1-\frac{e_2}{q})  + \frac{n}{e_3}(1-\frac{e_3}{r}), $$
which gives \ 
$\displaystyle{n < n+2 = n\left (\frac{1}{e_1}  + \frac{1}{e_2}  + \frac{1}{e_3} \right ) \ \hbox{and thus} \ \ \frac{1}{e_1}  + \frac{1}{e_2}  + \frac{1}{e_3} > 1.}$ 
\smallskip

Next, from Lemma~\ref{lem:abGamma}  we conclude that  $\Gamma/\Gamma'$ is a smooth non-trivial abelian quotient of $\Delta^+(e_1,e_2,e_3)$, 
and hence we find  there are only two possibilities, 
namely that either $(e_1,e_2,e_3) = (2,2,2)$, with  $\Gamma/\Gamma' \cong C_2 \times C_2$, 
or $(e_1,e_2,e_3)$ is some permutation of $(1,k,k)$ for some $k$, with $\Gamma/\Gamma' \cong C_k$. 
(In the other cases, such as $(e_1,e_2,e_3) = (2,3,3)$, there is no smooth abelian quotient.) 

In the first case, where $\Gamma/\Gamma' \cong C_2 \times C_2$, we find $n = |\Gamma:\Gamma'|= 4$ 
and $n_i = \frac{n}{e_i} = \frac{4}{2} =1$ for each $i$, 
a contradiction because at least one of the $n_i$ should be $1$ in Cases (4), (6) and (7). 

In the second case, where $(e_1,e_2,e_3)$ is some permutation of $(1,k,k)$ for some $k$, 
with $\Gamma/\Gamma' \cong C_k$, 
we have $n = k$, and we may suppose without loss of generality that $e_1 = e_2 = n$ and $e_3 = 1$, 
which gives $n_1 = n_2 = 1$ and $n_3 = n$.
Then since at most one of the $m_i$ is $1$, and at least one of the $n_i$ is $1$, 
we can further suppose without loss of generality that $(m_1,m_2,m_3) = (\frac{p}{n},\frac{q}{n},r)$, 
with at most one of $\frac{p}{n}$ and $\frac{q}{n}$ being $1$.
Hence we can assume that $\Gamma'$ has compact signature $(0;\,m_2,m_3^{(n_3)})$ or $(0;\,m_1,m_2,m_3^{(n_3)})$, 
which after a permutation of $(m_1,m_2,m_3)$ if necessary, puts $\Gamma$ into Case (4) or (6), but not Case (7). 

In particular, in Case (7) there are always infinitely many smooth finite soluble quotients with derived length $c = 3$,

We can treat these two compact signatures $(0;\,m_2,m_3^{(n_3)})$ or $(0;\,m_1,m_2,m_3^{(n_3)})$ together, 
by just taking the first as a special case of the second, with $m_1 = \frac{p}{k} = 1$, and following the situation for Case (6).
Here $N = |\Gamma':K_m(\Gamma')| = m_3^{\,n_3-1} = r^{\,n-1}$, 
and if $m = m_3$ then $K_m(\Gamma')$ has compact signature $(g';\, m_1^{(N)},m_2^{(N)})$ for some $g'$. 

Now in the proof of Theorem~\ref{thm:main}, if $g' > 0$ then in a similar way to earlier, 
there are infinitely many possible choices for a positive integer $m'$ giving rise to a smooth finite soluble 
quotient $\Gamma/L$ of $\Gamma$ with derived length $3$ for which $|K_m(\Gamma')/L| = (m')^{2g'} m_1^{\,N-1}m_2^{\,N-1}$, 
and hence infinitely many such quotients.  

So for the rest of the current proof, we may suppose that $g' = 0$, with $K_m(\Gamma')$ having compact signature $(0;\, m_1^{(N)},m_2^{(N)})$ 
where $N = r^{\,n-1}$, and that $(m_1,m_2) = (\frac{p}{k},\frac{q}{k})$. 
Then 
$$nN\left (1 - \frac{1}{p} - \frac{1}{q} - \frac{1}{r} \right ) = nN\mu(\Gamma) = \mu(K_m(\Gamma')) 
= 0-2 + N(1-\frac{n}{p})  + N(1-\frac{n}{q}), $$
which gives \ 
$\displaystyle{nN - \frac{nN}{r} = 2N -2 \ \ \hbox{and thus} \ \ 2 = N \left ( 2 - n + \frac{n}{r} \right ) = r^{n-2} ( 2r - nr + n ).}$ 
\smallskip

It follows that $r^{n-2} = 1$ or $2$ and hence either $n =  2$, or $n = 3$ and $r = 2$. 

In the former case, $k = n = 2$ and $1 = e_3 = \gcd(r, \lcm(p,q))$ so $r$ is coprime to both $p$ and $q$, 
and then $2 = k = e_1 = \gcd(p,\lcm(q,r)) = \gcd(p,q)$, and also $p \ne q$ for otherwise $p = q = 2$ 
but then $\Gamma = \Delta^+(2,2,r)$ is not hyperbolic. 
Conversely, if $r$ is coprime to both $p$ and $q$, and $p \ne q$, and $\gcd(p,q) = 2$, 
then $(e_1,e_2,e_3) = (2,2,1)$ and $n = k = 2$, so $r^{n-2} ( 2r - nr + n ) = 2$ and then $g' = 0$.

(Furthermore, if $p = 2$ or $q = 2$ then $\frac{p}{k} = 1$ or $\frac{q}{k} = 1$ and hence Case (4) applies, 
but otherwise $\frac{p}{k} > 1$ and $\frac{q}{k} > 1$ and hence Case (6) applies.) 

Finally, in the latter case, $r = 2$ and $k = n = 3$ and again $1 = e_3 = \gcd(r, \lcm(p,q))$ so $r$ is coprime to both $p$ and $q$, 
and then $3 = k = e_1 = \gcd(p,\lcm(q,r)) = \gcd(p,q)$, and also $p \ne q$ for otherwise $p = q = 3$ 
but then $\Gamma = \Delta^+(3,3,2)$ is not hyperbolic. 
Conversely, if $r = 2$ is coprime to both $p$ and $q$, and $p \ne q$, and $\gcd(p,q) = 3$, 
then $(e_1,e_2,e_3) = (3,3,1)$ and $n = k = 3$, so $r^{n-2} ( 2r - nr + n ) = 2(4-6+3) = 2$ and so $g' = 0$.

(Furthermore, if $p = 3$ or $q = 3$ then $\frac{p}{k} = 1$ or $\frac{q}{k} = 1$ and hence Case (4) applies, 
but otherwise $\frac{p}{k} > 1$ and $\frac{q}{k} > 1$ and hence Case (6) applies.) 
\end{proof}

Finally, we note that the special cases considered by Chetiya et al in~\cite{ChetiyaDuttaPatra} all lie in Case (4). 
For those we may take $(p,q,r) = (m,\ell u,\ell)$ with $m$ coprime to each of $\ell$ and $u$,  
and then $\Gamma'$ has compact signature $(0;\, u,m^{(\ell)})$, and $K_m(\Gamma')$ has compact signature $(g;\, u^{(N)})$ where $N = m^{\ell-1}$ 
and $g = 1 + \frac{1}{2}\,m^{\ell-2}\,(\ell m -\ell - 2m)$. 
Hence if $\ell = 2$ or if $\ell = 3$ and $m = 2$, then $g = 0$ and there are only finitely many smooth finite soluble quotients with derived length $3$, 
while otherwise $\ell m - \ell - 2m \ge 0$ and then $g > 0$ so there are infinitely many smooth finite soluble quotients with derived length $3$. 

\bigskip\bigskip

\noindent 
{\Large\bf Acknowledgements}

\medskip
We are grateful to the N.Z. Marsden Fund (via grant UOA2320) and the University of Auckland 
for supporting research that helped produce this paper, 
and to Javier Cirre, Gareth Jones, Gabriel Verret and Bradley Windleborn for some helpful discussions. 
Also we are happy to acknowledge extensive use of the {\sc Magma} system \cite{Magma} in 
experimental observations and partial confirmation of conjectures arising from those.  
Finally, we are grateful to the two referees for their helpful suggestions for improvements 
to the exposition of this paper.


\end{document}